\documentclass[10pt]{article}

\usepackage{etex}
\usepackage[english]{babel}
\usepackage[utf8]{inputenc}
\usepackage[T1]{fontenc}
\usepackage{lmodern}
\usepackage{authblk}

\usepackage{amsmath, latexsym, amssymb, amsfonts, amsthm}
\usepackage[colorlinks,citecolor=black,filecolor=black,linkcolor=black,urlcolor=black ]{hyperref}
\usepackage{graphicx}
\usepackage{subfigure} 

\def\ra {\rightarrow}

\def\be{\begin{equation}}   \def\ee{\end{equation}}
\def\ba   {\begin{array}}      \def\ea   {\end{array}}
\def\bea  {\begin{eqnarray}}   \def\eea  {\end{eqnarray}}
\def\bean {\begin{eqnarray*}}  \def\eean {\end{eqnarray*}}

\newtheorem{theorem} {Theorem}
\newtheorem{lemma}{Lemma}
\newtheorem{definition} {Definition}

\newtheorem{corollary} {Corollary}

\newcommand{\bo} {\ensuremath{{\bf i_1}}}
\newcommand{\bos}{\ensuremath{{\bf i_1^{\text 2}}}}
\newcommand{\bts}{\ensuremath{{\bf i_2^{\text 2}}}}

\newcommand{\bj}{\ensuremath{{\bf j}}}

\newcommand {\bjp}{\ensuremath{{\bf j_1}}}
\newcommand {\bjps}{\ensuremath{{\bf j_1^{\text 2}}}}
\newcommand {\bjd}{\ensuremath{{\bf j_2}}}
\newcommand {\bjds}{\ensuremath{{\bf j_2^{\text 2}}}}

\newcommand{\bt} {\ensuremath{{\bf i_2}}}
\newcommand{\bb} {\ensuremath{{\bf i_3}}}
\newcommand{\bbs} {\ensuremath{{\bf i_3^{\text 2}}}}
\newcommand{\bq} {\ensuremath{{\bf i_4}}}
\newcommand{\bqs} {\ensuremath{{\bf i_4^{\text 2}}}}

\newcommand{\bk} {\ensuremath{{\bf i_{k}}}}
\newcommand {\bjt}{\ensuremath{{\bf j_3}}}
\newcommand {\bjts}{\ensuremath{{\bf j_3^{\text 2}}}}
\newcommand{\bil} {\ensuremath{{\bf i_l}}}
\newcommand{\bim} {\ensuremath{{\bf i_m}}}

\newcommand{\ett} {\ensuremath{\gamma_2}}
\newcommand{\etc} {\ensuremath{\Ol{\gamma}_2}}

\newcommand{\Ol}{\overline}

\newcommand{\mC}{\ensuremath{\mathbb{C}}}
\newcommand{\mD}{\ensuremath{\mathbb{D}}}
\newcommand{\mN}{\ensuremath{\mathbb{N}}}

\newcommand{\mR}{\ensuremath{\mathbb{R}}}
\newcommand{\mT}{\ensuremath{\mathbb{T}}}

\newcommand{\mM}{\ensuremath{\mathbb{M}}}
\newcommand{\mMan}{\ensuremath{\mathcal{M}}}
\newcommand{\mTet}{\ensuremath{\mathcal{T}}}

\newcommand{\op}{\left(}
\newcommand{\fp}{\right)}
\newcommand{\oa}{\left\{}
\newcommand{\fa}{\right\}}
\newcommand{\lc}{\left[}
\newcommand{\rc}{\right]}

\newcommand{\Qpit}{\ensuremath{\left\lbrace Q_{p,c}^m(0) \right\rbrace_{m=1}^{\infty}}}

\newcommand{\mManp}{\ensuremath{\mathcal{M}^p}}

\newcommand{\mHt}{\ensuremath{\mathcal{H}^2}}

\newcommand{\mHyb}{\ensuremath{\mathcal{H}}}
\newcommand{\cP}{\ensuremath{\mathcal{P}}}

\newcommand{\bHpc}{\ensuremath{\mathbf{H}_{p,c}}}

\title{Tricomplex dynamical systems generated by polynomials of even degree}

\author[1]{Pierre-Olivier Paris\'e\thanks{E-mail: {\tt pierre-olivier.parise@uqtr.ca}}}
\author[2]{Thomas Ransford \thanks{E-mail: {\tt ransford@mat.ulaval.ca}}}
\author[1]{Dominic Rochon\thanks{E-mail: {\tt dominic.rochon@uqtr.ca}}}
\affil[1]{Département de mathématiques et d'informatique, Université du Québec\\
 C.P. 500, Trois-Rivières, Québec, Canada, G9A 5H7\\}
\affil[2]{Département de mathématiques et de statistique, Université Laval\\
1045, av. de la Médecine, Québec, Canada, G1V 0A6.}

\begin{document}
\maketitle

\begin{abstract}
In this article, we give the exact interval of the cross section of the \textit{Multibrot} sets generated by the polynomial $z^p+c$ where $z$ and $c$ are complex numbers and $p \geq 2$ is an even integer. Furthermore, we show that the same Multibrots defined on the hyperbolic numbers are always squares. Moreover, we give a generalized 3D version of the hyperbolic Multibrot set and prove that our generalization is an octahedron for a specific 3D slice of the tricomplex polynomial $\eta^p+c$ where $p \geq 2$ is an even integer.
\end{abstract}\vspace{0.5cm}

\section*{Introduction}
Multicomplex dynamics appears for the first time in 2000 (see \cite{Rochon1} and \cite{Rochon2}). The author of these articles used a commutative generalization of complex numbers called the bicomplex numbers, denoted $\mM (2)$, $\mathbb{BC}$ or $\mT$, to extend the well known Mandelbrot set in four dimensions and to give a 3D version of it.

Another way to generalize the Mandelbrot set is to consider the Multibrot sets (also called Mandelbrot set). In \cite{Gujar}, \cite{Papathomas}, \cite{RochonParise} and \cite{Chine1}, a Multibrot set of order $p$ is defined as
	\begin{align*}
	\mManp := \oa c \in \mC \, : \, \oa Q_{p,c}^n(0) \fa_{n=1}^{\infty} \text{ is bounded } \fa
	\end{align*}
where $Q_{p,c}(z) = z^p + c$ with $z,c \in \mC$ and $p \geq 2$ is an integer. For $p = 2$, this is exactly the classical Mandelbrot set. It is well known that the Mandelbrot set is connected and it crosses the real axis on the interval $[-2, \frac{1}{4} ]$. Moreover, from \cite{RochonParise} and \cite{PariseRochon2}, we know that 
\begin{equation}
	\mMan^p \cap \mR = \lc -\frac{p-1}{p^{p/(p-1)}}, \frac{p-1}{p^{p/(p-1)}} \rc 
	\end{equation}
for any odd integer $p > 2$. The proof was based on a precise analysis of the roots of the polynomial $g_c(z):=z^p - z +c$. However, the following characterization 
	\begin{align}
	\mManp \cap \mR = \lc -2^{1/(p-1)}, \frac{p-1}{p^{p/(p-1)}} \rc \label{EvenCharacterization}
	\end{align}
for an even integer $p \geq 2$ was left unresolved in \cite{PariseRochon2}. In this article, we use a new approach to prove that the Conjecture \eqref{EvenCharacterization} is true and has many consequences in tricomplex dynamics.

The article is separated into three sections. In the first section, we recall some basics of the theory of tricomplex numbers. In the second section, we give the proof of the main theorem related with the Conjecture \eqref{EvenCharacterization}. In the third section, we prove that the Hyperbrot set of order $p$ defined as
	\begin{equation}
	\mHyb^p:=\oa c \in \mD \, : \, \oa Q_{p,c}^m(0)\fa_{m=1}^{\infty } \text{ is bounded} \fa \label{HyperbrotDef}
	\end{equation}
where $\mD$ is the set of hyperbolic numbers (see \cite{RochonShapiro}, \cite{vajiac2} and \cite{Sobczyk}) is a square when the degree $p \geq 2$ is an even integer. Finally, we define the generalized 3D version of the Hyperbrot sets and prove that they are regular octahedrons.


\section{Preliminaries}\label{Premilinaries}
In this section, we begin by introducing the tricomplex space $\mM (3)$. One may refer to \cite{Baley}, \cite{GarantRochon}, \cite{GarantPelletier}, \cite{Parise} and \cite{Vajiac} for more details on the next properties.

A tricomplex number $\eta$ is composed of two coupled bicomplex numbers $\zeta_1$, $\zeta_2$ and an imaginary unit $\bb$ such that
	\begin{equation}
	\eta=\zeta_1 + \zeta_2 \bb \label{eq2.1}
	\end{equation}
where $\bbs=-1$. The set of such tricomplex numbers is denoted by $\mM (3)$. Since $\zeta_1,\zeta_2 \in \mM (2)$, we can write them as $\zeta_1=z_1+ z_2\bt$ and $\zeta_2=z_3+ z_4\bt$ where $z_1,z_2,z_3,z_4 \in \mM (1)\simeq \mC$. In that way, \eqref{eq2.1} can be rewritten as
	\begin{equation}
	\eta=z_1+ z_2 \bt+  z_3 \bb+  z_4 \bjt\label{eq2.2}
	\end{equation}
where $\bts=-1$, $\bt \bb = \bb \bt = \bjt$ and $\bjts=1$. Moreover, as $z_1$, $z_2$, $z_3$ and $z_4$ are complex numbers (in $\bo$), we can write the number $\eta$ in a third form as
	\begin{align}
	\eta&=a+ b\bo + (c+  d\bo)\bt + (e +  f\bo)\bb + (g +  h\bo)\bjt\notag\\
	&=a+ b\bo +  c\bt +  d\bjp +  e\bb +  f\bjd + g \bjt +  h\bq\label{eq2.3}
	\end{align}
where $\bos=\bqs=-1$, $\bq =\bo \bjt = \bo \bt \bb$, $\bjd = \bo \bb = \bb \bo$, $\bjds=1$, $\bjp=\bo \bt = \bt \bo$ and $\bjps=1$. After ordering each term of \eqref{eq2.3}, we get the following representations of the set of tricomplex numbers:
	\begin{align}
	\mM (3) &:= \oa \eta = \zeta_1 +  \zeta_2\bb \, : \, \zeta_1, \zeta_2 \in \mM (2) \fa \notag\\
&=\oa z_1+ z_2 \bt+  z_3 \bb+ z_4 \bjt  \, : \, z_1,z_2,z_3,z_4 \in \mM (1) \fa \notag\\
&=\oa x_0+ x_1\bo + x_2\bt  +  x_3\bb + x_4\bq  +  x_5\bjp +  x_6 \bjd+  x_7\bjt \, \right. \notag \\ & \qquad \qquad \left. : \, x_i \in \mM (0)=\mR \text{ for } i=0,1,2, \ldots , 7 \fa \text{.} \label{EqRep}
	\end{align}
Let $\eta_1=\zeta_1+\zeta_2\bb$ and $\eta_2=\zeta_3 + \zeta_4\bb$ be two tricomplex numbers with $\zeta_1,\zeta_2,\zeta_3,\zeta_4 \in \mM (2)$. We define the equality, the addition and the multiplication of two tricomplex numbers as
	\begin{align}
	\eta_1&=\eta_2 \text{ iff } \zeta_1=\zeta_3 \text{ and } \zeta_2=\zeta_4 \label{eq2.4}\\
\eta_1 + \eta_2 &:= (\zeta_1 + \zeta_3) + (\zeta_2+\zeta_4)\bb \label{eq2.5}\\
\eta_1 \cdot \eta_2&:= (\zeta_1\zeta_3-\zeta_2\zeta_4)+(\zeta_1\zeta_4 + \zeta_2\zeta_3)\bb \label{eq2.6}\text{.}
	\end{align}
Table \ref{tabC1} shows the results after multiplying each tricomplex imaginary unity two by two. The set of tricomplex numbers with addition $+$ and multiplication $\cdot$ forms a commutative ring with zero divisors.
	\begin{table}
	\centering
		\begin{tabular}{c|*{9}{c}}
		$\cdot$ & 1 & $\mathbf{i_1}$ & $\mathbf{i_2}$ & $\mathbf{i_3}$ & $\mathbf{i_4}$ & $\bjp$ & $\bjd$ & $\mathbf{j_3}$\\\hline
1 & 1 & $\mathbf{i_1}$ & $\mathbf{i_2}$ & $\mathbf{i_3}$ & $\mathbf{i_4}$ & $\bjp$ & $\bjd$ & $\mathbf{j_3}$\\
$\mathbf{i_1}$ & $\mathbf{i_1}$ & $-\mathbf{1}$ & $\bjp$ & $\bjd$ & $-\mathbf{j_3}$ & $-\mathbf{i_2}$ & $-\mathbf{i_3}$ & $\mathbf{i_4}$\\
$\mathbf{i_2}$ & $\mathbf{i_2}$ & $\bjp$ & $-\mathbf{1}$ & $\mathbf{j_3}$ & $-\bjd$ & $-\mathbf{i_1}$ & $\mathbf{i_4}$ & $-\mathbf{i_3}$\\
$\mathbf{i_3}$ & $\mathbf{i_3}$ & $\bjd$ & $\mathbf{j_3}$ & $-\mathbf{1}$ & $-\bjp$ & $\mathbf{i_4}$ & $-\mathbf{i_1}$ & $-\mathbf{i_2}$\\
$\mathbf{i_4}$ & $\mathbf{i_4}$ & $-\mathbf{j_3}$  & $-\bjd$ & $-\bjp$ & $-\mathbf{1}$ & $\mathbf{i_3}$ & $\mathbf{i_2}$ & $\mathbf{i_1}$\\
$\bjp$ & $\bjp$ & $-\mathbf{i_2}$  & $-\mathbf{i_1}$ & $\mathbf{i_4}$ & $\mathbf{i_3}$ & $\mathbf{1}$ & $-\mathbf{j_3}$ & $-\bjd$\\
$\bjd$ & $\bjd$ & $-\mathbf{i_3}$  & $\mathbf{i_4}$ & $-\mathbf{i_1}$ & $\mathbf{i_2}$ & $-\mathbf{j_3}$ &  $\mathbf{1}$ & $-\bjp$\\
$\mathbf{j_3}$ & $\mathbf{j_3}$ & $\mathbf{i_4}$ &$-\mathbf{i_3}$  & $-\mathbf{i_2}$ & $\mathbf{i_1}$ & $-\bjd$ & $-\bjp$ &  $\mathbf{1}$ \\
		\end{tabular}
	\caption{Products  of tricomplex imaginary units}\label{tabC1}
	\end{table}

A tricomplex number has a useful representation using the idempotent elements $\ett =\frac{1+\bjt}{2}$ and $\etc =\frac{1-\bjt}{2}$. Recalling that $\eta = \zeta_1 +  \zeta_2\bb$ with $\zeta_1, \zeta_2 \in \mM (2)$, the idempotent representation of $\eta$ is
	\begin{equation}
	\eta = (\zeta_1- \zeta_2\bt)\ett + (\zeta_1+ \zeta_2\bt)\etc \label{eq2.7}\text{.}
	\end{equation}
The representation \eqref{eq2.7} of a tricomplex number allows to add and multiply tricomplex numbers term-by-term. In fact, we have the following theorem (see \cite{Baley}):
	\begin{theorem}\label{theo2.2}
	Let $\eta_1=\zeta_1 +  \zeta_2\bb$ and $\eta_2=\zeta_3 +  \zeta_4\bb$ be two tricomplex numbers. Let $\eta_1=u_1\ett + u_2 \etc$ and $\eta_2=u_3\ett + u_4\etc$ be the idempotent representation \eqref{eq2.7} of $\eta_1$ and $\eta_2$. Then,
	\begin{enumerate}
\item $\eta_1+\eta_2=(u_1+u_3)\ett + (u_2+u_4)\etc$;
\item $\eta_1 \cdot \eta_2 = (u_1 \cdot u_3)\ett + (u_2 \cdot u_4)\etc$;
\item $\eta_1^m=u_1^m \ett + u_2^m \etc$ $\forall m \in \mN$.
	\end{enumerate}
	\end{theorem}

Moreover, we define the norm $\Vert \cdot \Vert_3 :\, \mM (3) \rightarrow \mR$ of a tricomplex number $\eta=\zeta_1 + \zeta_2\bb$ as
	\begin{align}
	\Vert \eta \Vert_3 & := \sqrt{\Vert\zeta_1\Vert_2^2+\Vert\zeta_2\Vert_2^2}=\sqrt{\sum_{i=1}^2|z_i|^2+\sum_{i=3}^4|z_i|^2}\label{eq2.15}\\
&=\sqrt{\sum_{i=0}^7x_i^2}.\notag
	\end{align}
According to the Euclidean norm \eqref{eq2.15}, we say that a sequence $\oa s_m \fa_{m=1}^{\infty} $ of tricomplex numbers is bounded if and only if there exists a real number $M$ such that $\Vert s_m \Vert_3 \leq M$ for all $m \in \mN$. 

Finally, we recall (see \cite{GarantRochon} and \cite{RochonParise}) an important subset of $\mM (3)$ that is useful to section \ref{consequenceOfMainTheorem}.
	\begin{definition}\label{Tikilim}
	Let $\bk , \bil , \bim \in \oa 1, \bo , \bt , \bb , \bq , \bjp , \bjd , \bjt \fa$ with $\bk \neq \bil$, $\bk \neq \bim$ and $\bil \neq \bim$. We define a 3D subset of $\mM (3)$ as
		\begin{equation}
		\mT (\bim , \bk , \bil ):= \oa x_1\bk + x_2\bil + x_3\bim \, : \, x_1,x_2,x_3 \in \mR \fa \text{.}
		\end{equation}
	\end{definition}

\section{Proof of the main theorem}\label{proofOfMainTheo}
In this section, we show that the intersection of a Multibrot set with the real line is exactly an interval. Then, our aim is to prove the following result.
	\begin{theorem}\label{t2.3.1}
	Let $p$ be an even integer with $p\geq 2$. Then
	\begin{align}
	\mManp \cap \mR = \lc -2^{1/(p-1)}, \frac{p-1}{p^{p/(p-1)}} \rc .
	\end{align}
	\end{theorem}
	
	The following lemma will be useful.
	\begin{lemma}\label{lem3.1.1}
	Let $f : \mR \ra \mR$ be a continuous function such that $\liminf_{x \ra \infty} (f(x) - x) > 0$. Let $a := \sup \oa x \, : \, f(x) = x \fa$ (or $a := -\infty$ if $f$ has no fixed points). Then, for each $x > a$, we have $f^n(x) \ra \infty$ as $n \ra \infty$.
	\end{lemma}
	\begin{proof}
	Since $f(x) - x$ has no zeros in $(a, \infty )$ and $\liminf_{x \ra \infty} (f(x) - x) > 0$, we have $f(x) > x$ for all $x > a$. Thus, for each $x > a$, the sequence $f^n(x)$ is increasing. If it does not tend to infinity, then it must tend to a finite limit $b > a$, and letting $n \ra \infty$ in the relation $f (f^n(x)) = f^{n+1}(x)$, we get $f(b) = b$, contradicting the fact that $f(x) > x$ for all $x > a$. \hfill $\qed$
	\end{proof}
	
	\begin{proof}[of Theorem \ref{t2.3.1}]
	We first consider the case where $c \geq 0$. Observe that $Q_{p,c}$ is an increasing function on $[0, \infty )$. Thus, if $Q_{p,c}$ has a non-negative fixed point $a \geq 0$, then $Q_{p,c}$ maps the interval $[0,a]$ into itself, so $Q_{p,c}^n(0) \in [0, a]$ for all integers $n \geq 1$. On the other hand, if $Q_{p,c}$ has no fixed point $a \geq 0$, then by Lemma \ref{lem3.1.1}, $Q_{p,c}^n(0) \ra \infty$ as $n \ra \infty$. Thus, $c \in \mManp$ if and only if $Q_{p,c}$ has a fixed point $a$ with $a \geq 0$.
	
	Now, to say that $Q_{p,c}$ has a non-negative fixed point is equivalent to saying that $g_c(x) = x^p - x + c$ has a non-negative zero. Note that $g_c(x) \ra \infty$ as $x \ra \infty$, so $g_c$ has a non-negative zero if and only if $\min_{x \geq 0} g_c(x) \leq 0$. As $g_c'(0) = -1 < 0$, this minimum must be attained at a critical point of $g_c$. The critical points of $g_c$ are the solutions of $px^{p-1} - 1 = 0$ and the unique positive critical point is thus $x = (1/p)^{1/(p-1)}$. Therefore $\min_{x \geq 0} g_c(x) = g((1/p)^{1/(p-1)})$. In particular, we have
		\begin{align*}
		\min_{x \geq 0} g_c(x) \leq 0 &\iff \op \op \frac{1}{p} \fp^{p/(p-1)} - \op \frac{1}{p} \fp^{1/(p-1)} + c \fp \leq 0 \\
		& \iff c \leq (p-1) \op \frac{1}{p} \fp^{p/(p-1)} \text{.}
		\end{align*}
	To summarize, we have shown that, for $c \geq 0$, we have
		\begin{align}
		c \in \mManp \iff c \leq \frac{p-1}{p^{p/(p-1)}} \text{.}\label{firstPart}
		\end{align}
	
	Now we turn to the case $c < 0$. Letting $g_c(x) = x^p - x + c$ once again, we have $g_c(0) = c < 0$ and $g_c(x) \ra \infty$ as $x \ra \pm \infty$, so $g_c$ has at least two zeros, one positive and one negative. Moreover, $g'_c(x)$ has a unique root, so in fact $g_c$ has exactly two zeros, one positive and one negative. Then $Q_{p,c}$ has exactly two fixed points, one positive and one negative. Denote by $a$ the positive fixed point. If $Q_{p,c}(c) > a$, then by Lemma \ref{lem3.1.1}, $Q_{p,c}^n(Q_{p,c}(c)) \ra \infty$ as $n \ra \infty$, and thus $Q_{p,c}^n(0) = Q_{p,c}^{n-2}(Q_{p,c}(c)) \ra \infty$ as $n \ra \infty$. On the other hand, if $Q_{p,c}(c) \leq a$, then, since $Q_{p,c}$ is decreasing on $[c,0]$ and increasing on $[0,a]$, it maps the interval $[c,a]$ into itself, so $Q_{p,c}^n(0) \in [c, a]$ for all integers $n \geq 1$. Thus $c \in \mManp$ if and only if $Q_{p,c}(c) \leq a$.
	
	Now $x > a$ if and only if $x \geq 0$ and $Q_{p,c}(x) > x$. Therefore,
		\begin{align*}
		Q_{p,c}(c) > a & \iff Q_{p,c}(c) \geq 0 \text{ and } Q_{p,c}(Q_{p,c}(c)) > Q_{p,c}(c) \\
		& \iff c^p + c \geq 0 \text{ and } (c^p + c)^p + c > c^p + c \\
		& \iff c^p + c > -c \\
		& \iff c^{p-1} < -2 \\
		& \iff c < -2^{1/(p-1)} \text{.}
		\end{align*}
	Thus, for $c < 0$,
		\begin{align}
		c \in \mManp \iff c \geq -2^{1/(p-1)} \text{.}\label{secondPart}
		\end{align}
	Finally, combining \eqref{firstPart} and \eqref{secondPart}, we obtain the conclusion of the theorem.\hfill $\qed$
\end{proof}

Figure \ref{figureMultibrots} shows examples of some Multibrot sets. 
\begin{figure}
\centering
\subfigure[$\mMan^2$, $\lc -2,2\rc \times \lc -2,2 \rc$]{
	\includegraphics[scale = .23]{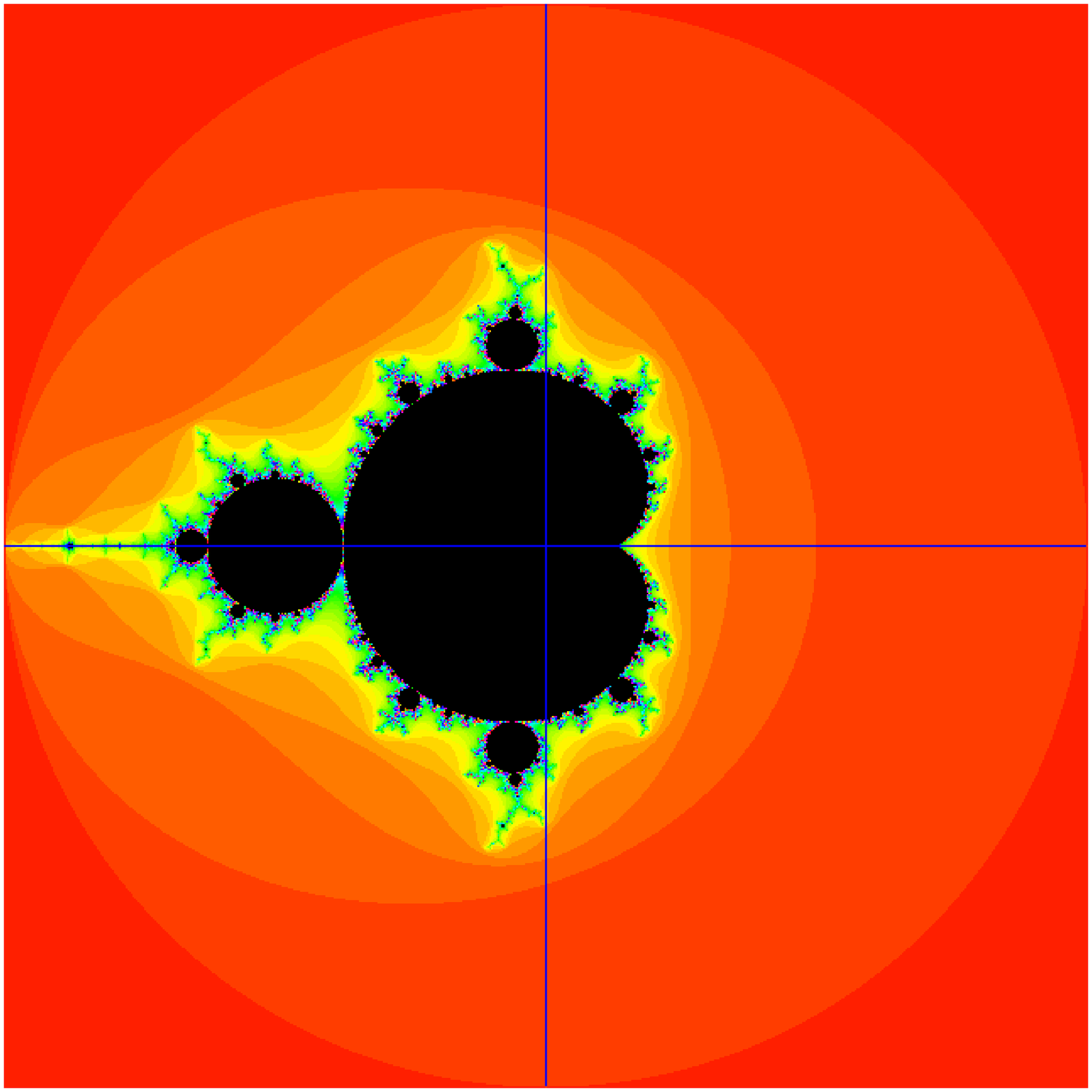}}
\subfigure[$\mMan^4$, $\lc -2^{\frac{1}{3}},2^{\frac{1}{3}}\rc \times \lc -2^{\frac{1}{3}},2^{\frac{1}{3}}\rc$]{
	\includegraphics[scale = .23]{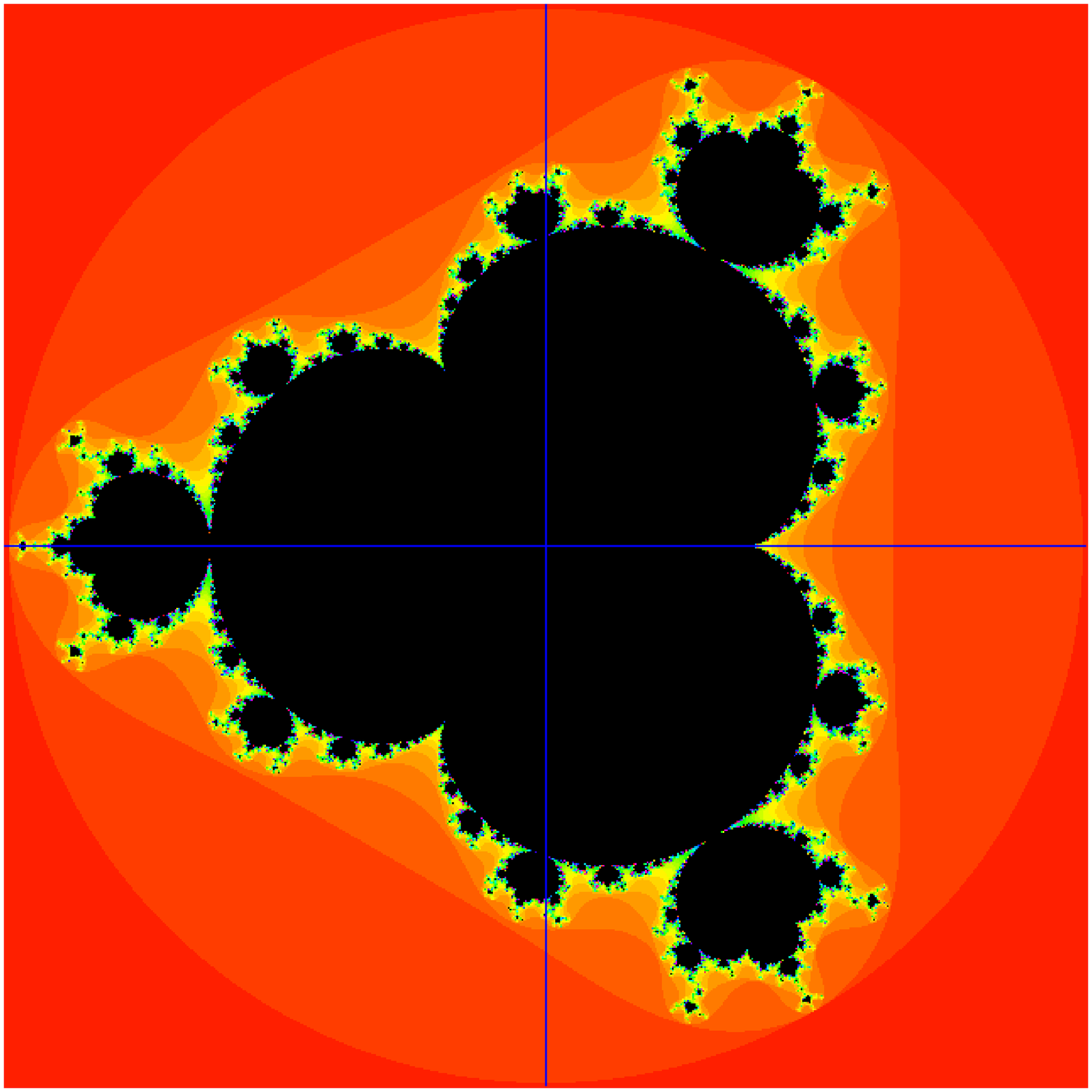}}
\subfigure[$\mMan^6$, $\lc -2^{\frac{1}{5}},2^{\frac{1}{5}}\rc \times \lc -2^{\frac{1}{5}},2^{\frac{1}{5}}\rc$]{
	\includegraphics[scale = .23]{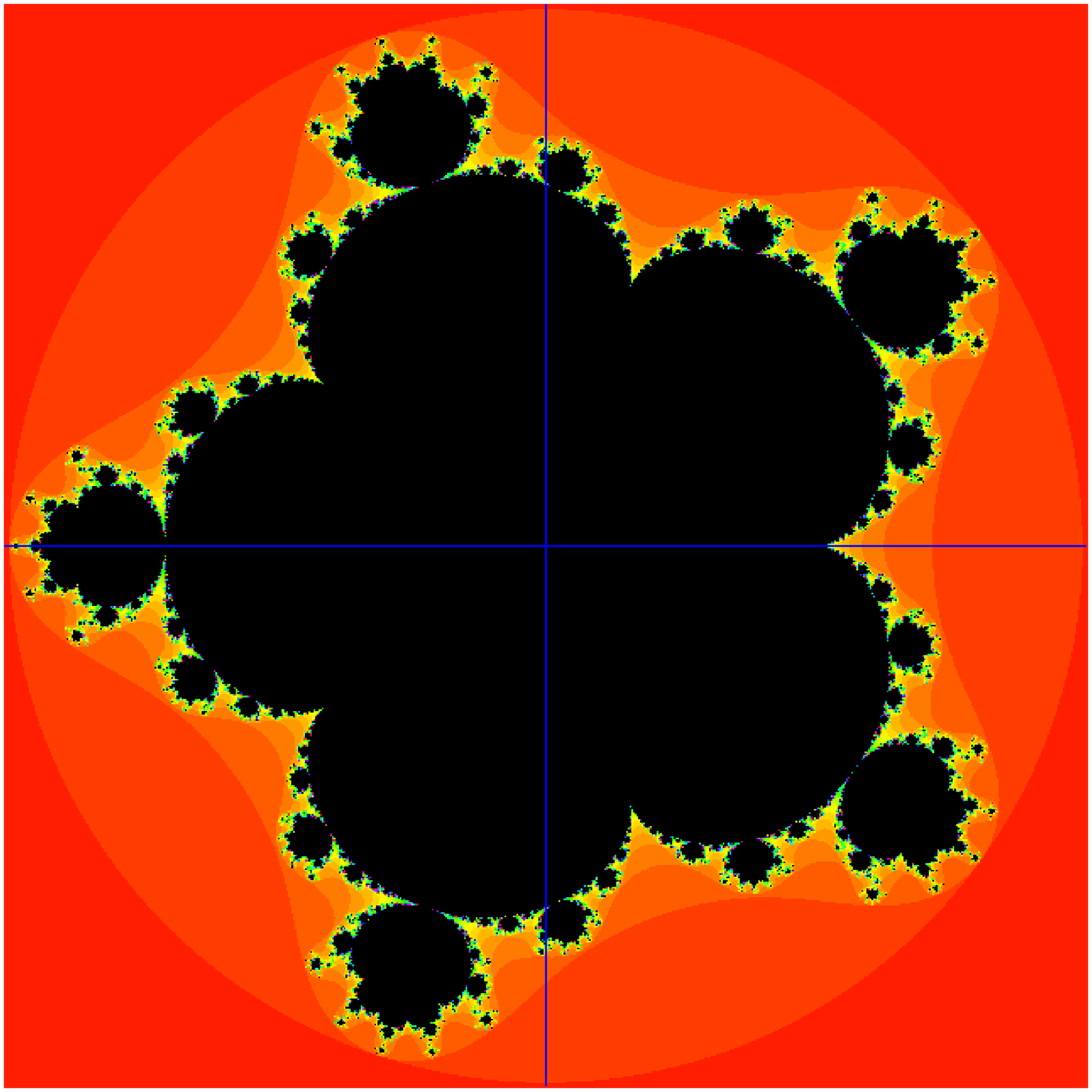}}
\subfigure[$\mMan^{12}$, $\lc -2^{\frac{1}{11}},2^{\frac{1}{11}}\rc \times \lc -2^{\frac{1}{11}},2^{\frac{1}{11}}\rc$]{
	\includegraphics[scale = 0.23]{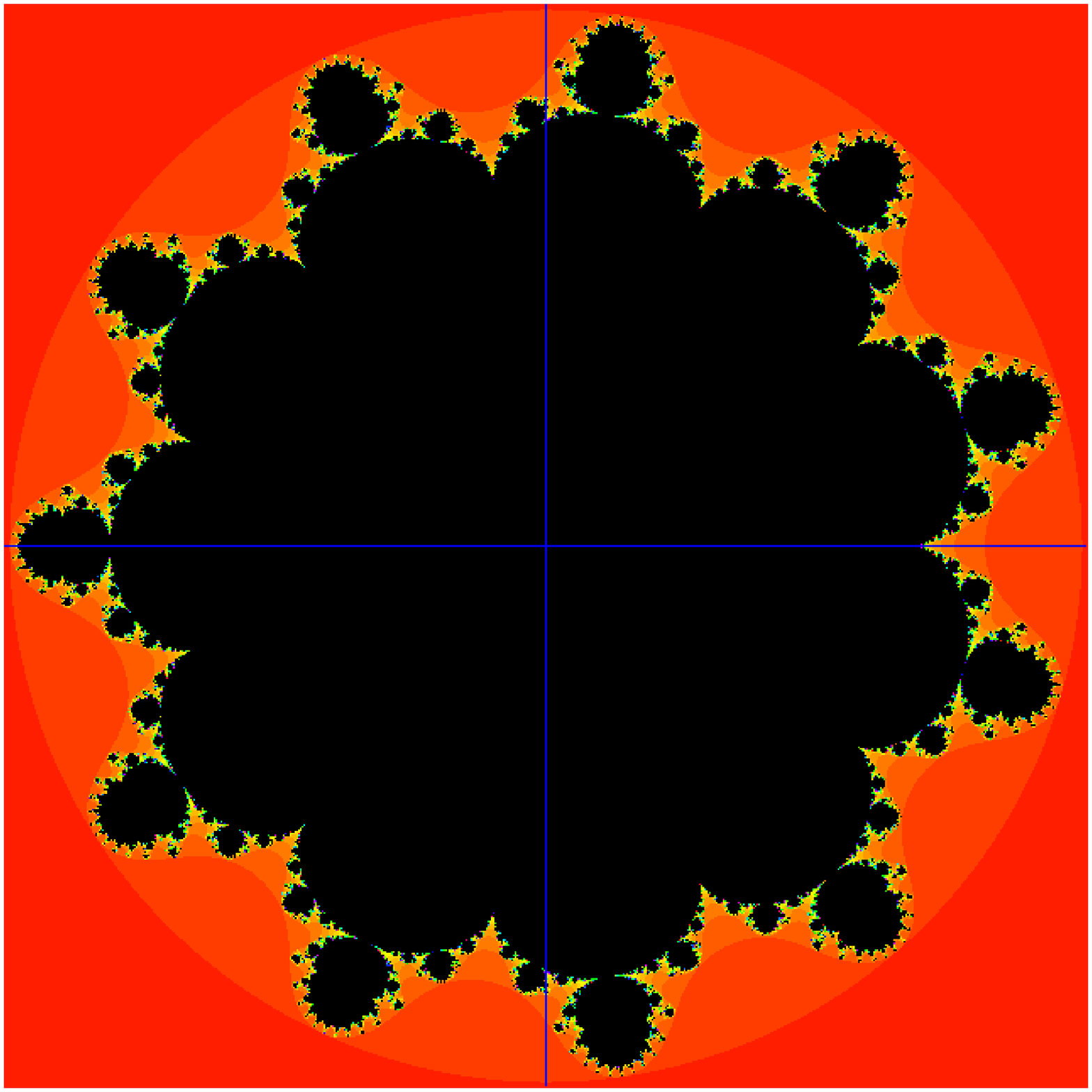}}
\caption{Multibrots for several even integers}\label{figureMultibrots}
\end{figure}

\section{Consequences of the main theorem}\label{consequenceOfMainTheorem}
In this section, we use Theorem \ref{t2.3.1} to prove that the Hyperbrots are always squares. In addition, we give a generalized 3D version of the Hyperbrot sets and prove that our generalization is a regular octahedron for even integers with $ p \geq 2$.

	\subsection{Characterization of the Hyperbrots}\label{sec3.4}
	In 1990, Senn \cite{Senn} generated the Mandelbrot set using the hyperbolic numbers. Instead of obtaining a fractal structure, the set obtained seemed to be a square. Four years later, in \cite{MET}, Metzler proved that $\mHt$ is precisely a square with diagonal length $2\frac{1}{4}$ and of side length $\frac{9}{8} \sqrt{2}$. It was also proved in \cite{PariseRochon2} that the Hyperbrots are always squares for odd integers greater than $2$. In this subsection, we generalized their results for even integers.
	
	According to the tools introduced in \cite{RochonParise} and Theorem \ref{t2.3.1}, we have the following result.
	\begin{theorem}\label{t3.2}
		Let $p\geq2$ be an even integer. Then, the Hyperbrot of order $p$ is characterized as
	\begin{equation*}
	\mHyb^p = \oa x+y\bj \in \mD \, : \, |x - t_p|+|y| \leq \frac{l_p}{2} \fa
	\end{equation*}
	where
		\begin{align*}
		t_p := \frac{ - p\, [(2p)^{1/(p-1)} - 1]-1}{2p^{p/(p-1)}} \quad \text{ and } \quad l_p := \frac{p\, [(2p)^{1/(p-1)} + 1] -1}{p^{p/(p-1)}} \text{.}
		\end{align*}
	\end{theorem}
	\begin{proof}
	Using the notations and Lemma 7 in \cite{RochonParise}, and the remark right after, $\oa \bHpc^m (\mathbf{0}) \fa _{m=1}^{\infty}$ is bounded if and only if the real sequences $\oa Q_{p,x-y}^m(0) \fa _{m=1}^{\infty}$ and $\oa Q_{p,x+y}^m(0) \fa _{m=1}^{\infty}$ are bounded. However, according to Theorem \ref{t2.3.1}, these sequences are bounded if and only if
		\begin{align*}
		-2^{1/(p-1)}\leq x-y \leq \frac{p-1}{p^{p/(p-1)}}
		\quad \text{ and } \quad
		 -2^{1/(p-1)} \leq x+y\leq \frac{p-1}{p^{p/(p-1)}}\text{.}
		\end{align*}
	Then, subtracting $t_p$ from both sides gives the following inequalities
	\begin{align}
	-\frac{l_p}{2} \leq x - t_p - y \leq \frac{l_p}{2}
	\quad \text{ and } \quad
	-\frac{l_p}{2} \leq x - t_p + y \leq \frac{l_p}{2}. \label{HypInequalities}
	\end{align}
	Moreover, inequalities \eqref{HypInequalities} are equivalent to
	\begin{align}
	|x - t_p| + |y| \leq \frac{l_p}{2} \text{.} \label{inequalityOfHyperbrot}
	\end{align}
	Thus, $c = x + y\bj \in \mHyb^p$ if and only if \eqref{inequalityOfHyperbrot} holds. \hfill $\qed$
	\end{proof}
	
Figure \ref{figureHyperbrots} represent faithfully Theorem \ref{t3.2}. We remark that each square is centered at the point $(t_p, 0)$. Moreover, we note that the squares seem to have a limit set as the even integer $p$ tends to infinity.
	
\begin{figure}
\centering
\subfigure[$\mHyb^2$, $\lc -2, 1 \rc \times \lc \frac{-3}{2}, \frac{3}{2} \rc$]{
	\includegraphics[scale = 0.15]{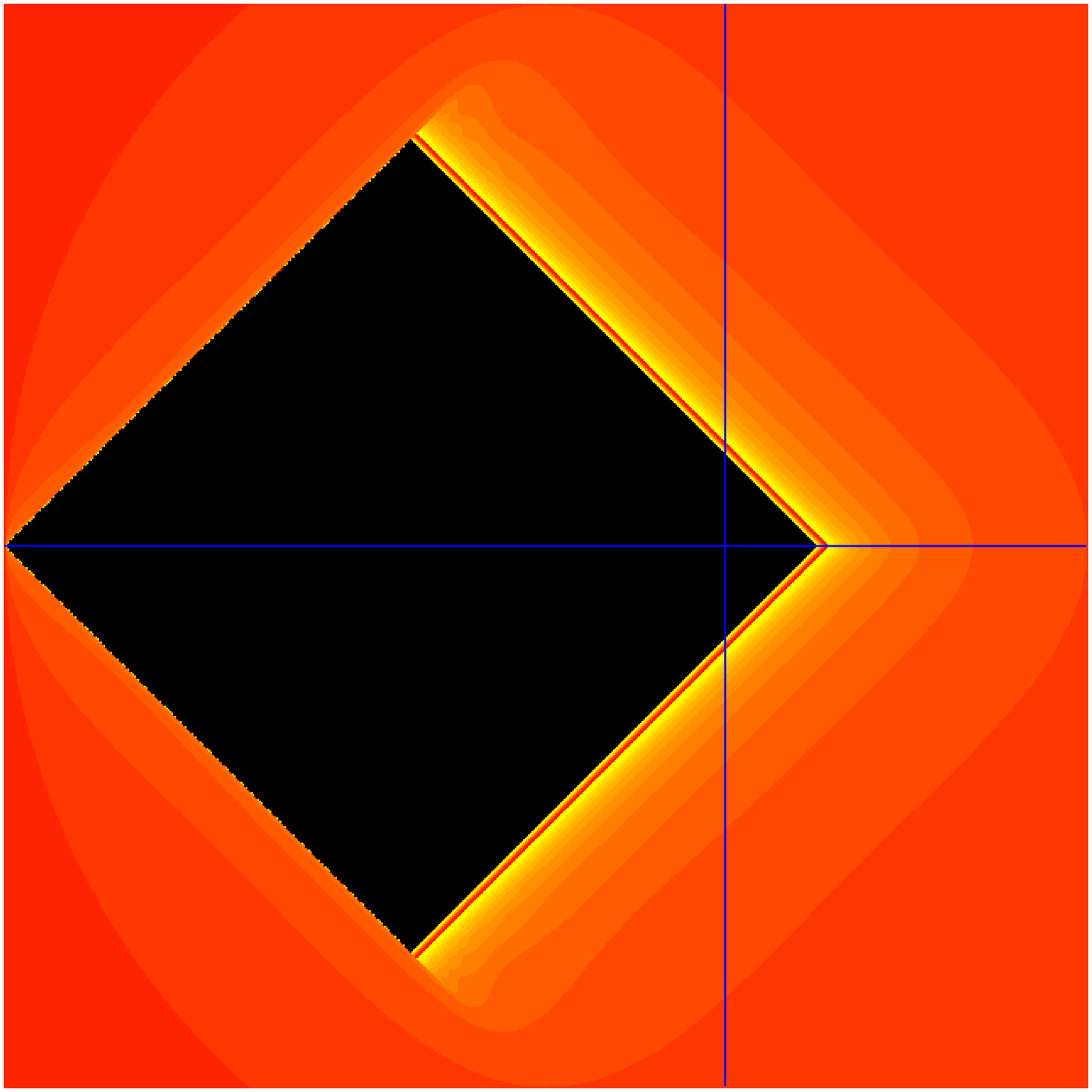}}
\subfigure[$\mHyb^4$, $\lc -2, 1 \rc \times \lc \frac{-3}{2}, \frac{3}{2} \rc$]{
	\includegraphics[scale = 0.15]{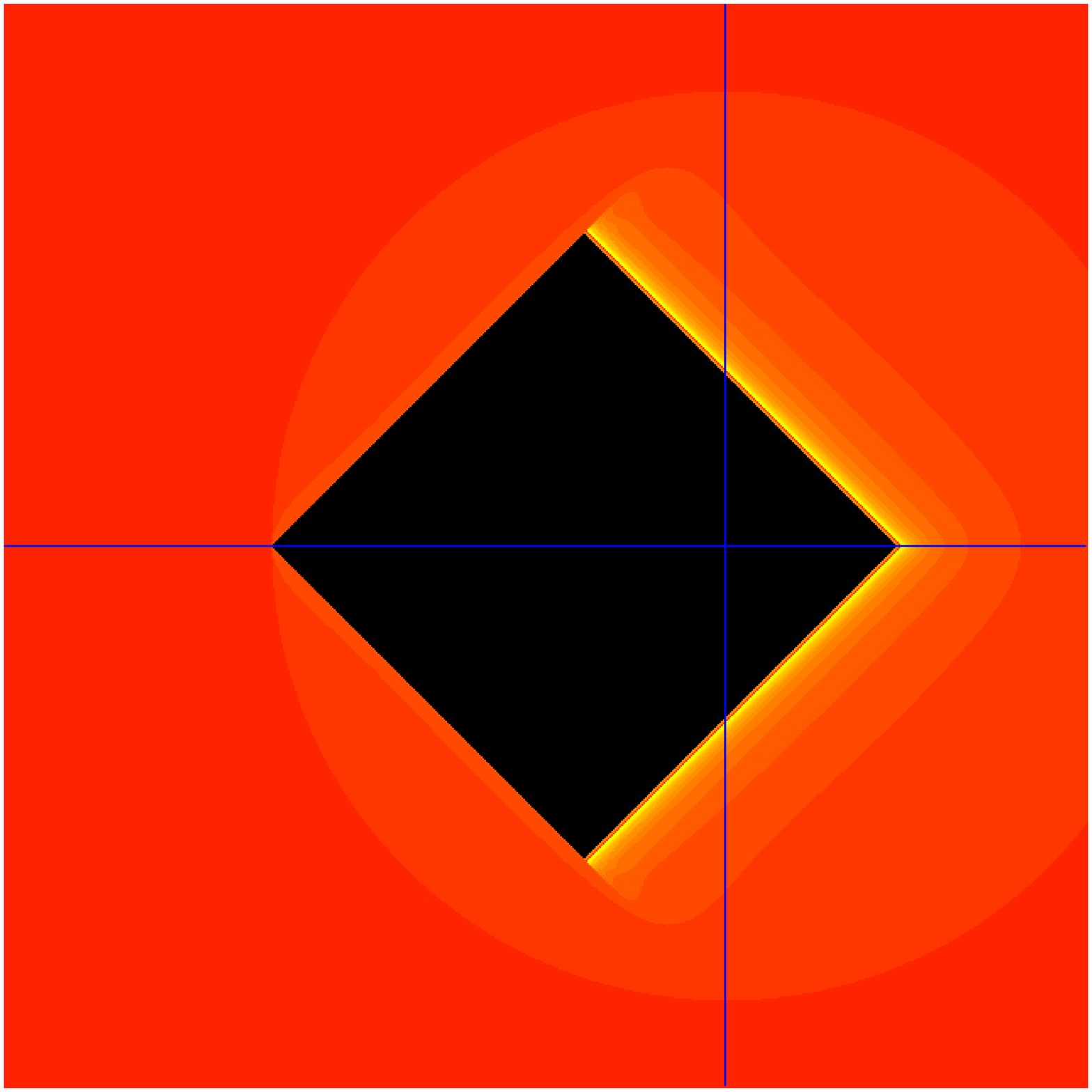}}
\subfigure[$\mHyb^8$, $\lc -2, 1 \rc \times \lc \frac{-3}{2}, \frac{3}{2} \rc$]{
	\includegraphics[scale = 0.15]{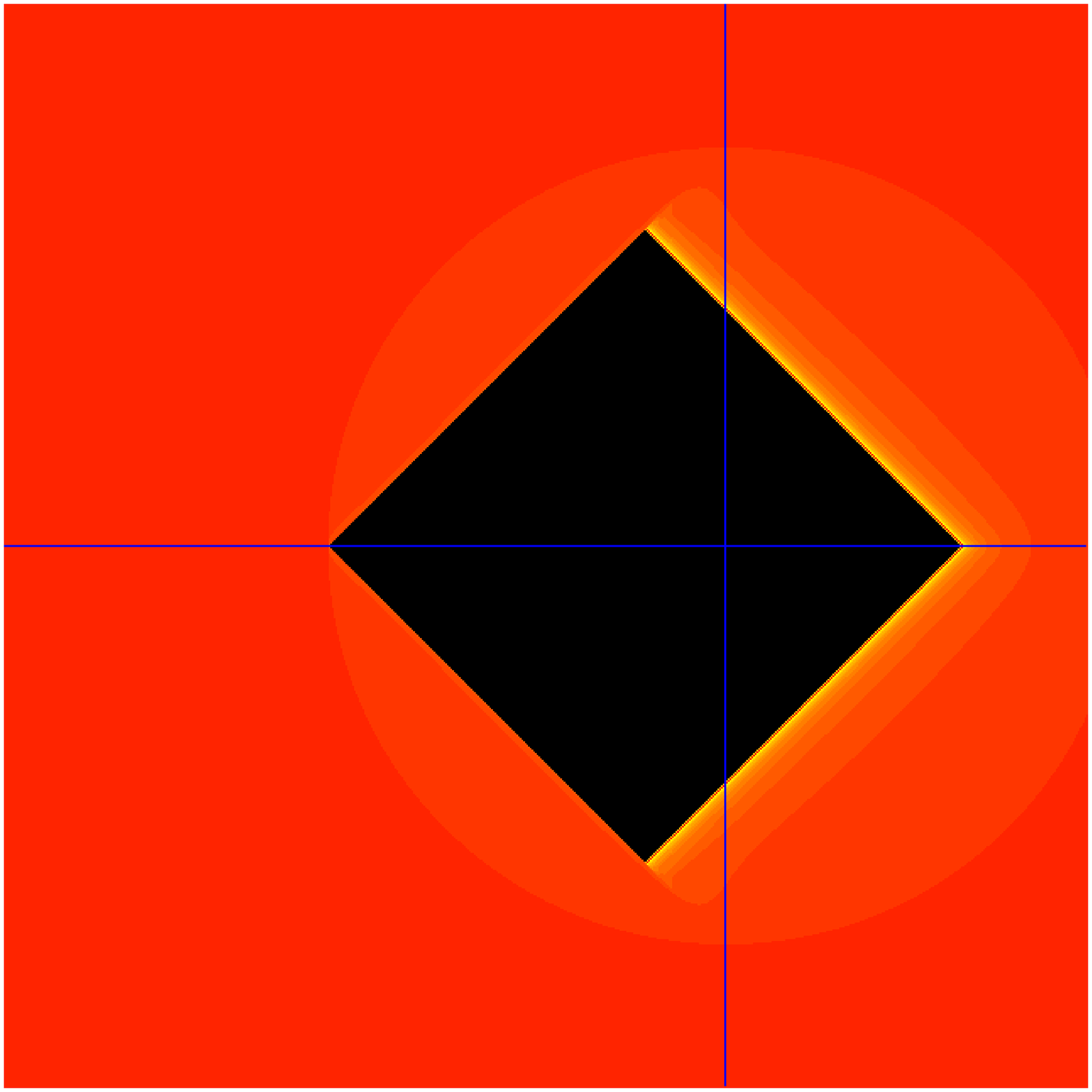}}
\subfigure[$\mHyb^{14}$, $\lc -2, 1 \rc \times \lc \frac{-3}{2}, \frac{3}{2} \rc$]{
	\includegraphics[scale = 0.15]{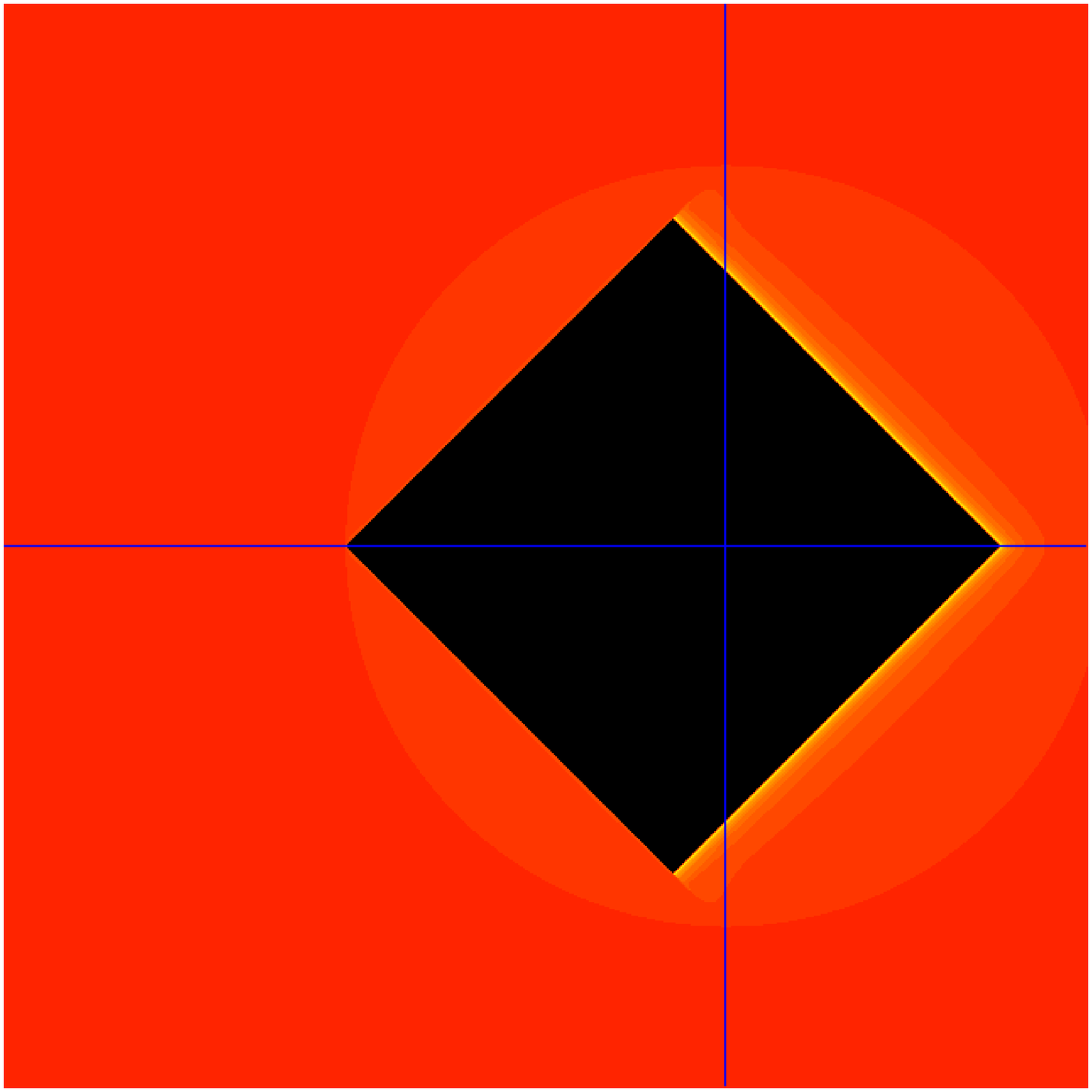}}
\subfigure[$\mHyb^{20}$, $\lc -2, 1 \rc \times \lc \frac{-3}{2}, \frac{3}{2} \rc$]{
	\includegraphics[scale = 0.15]{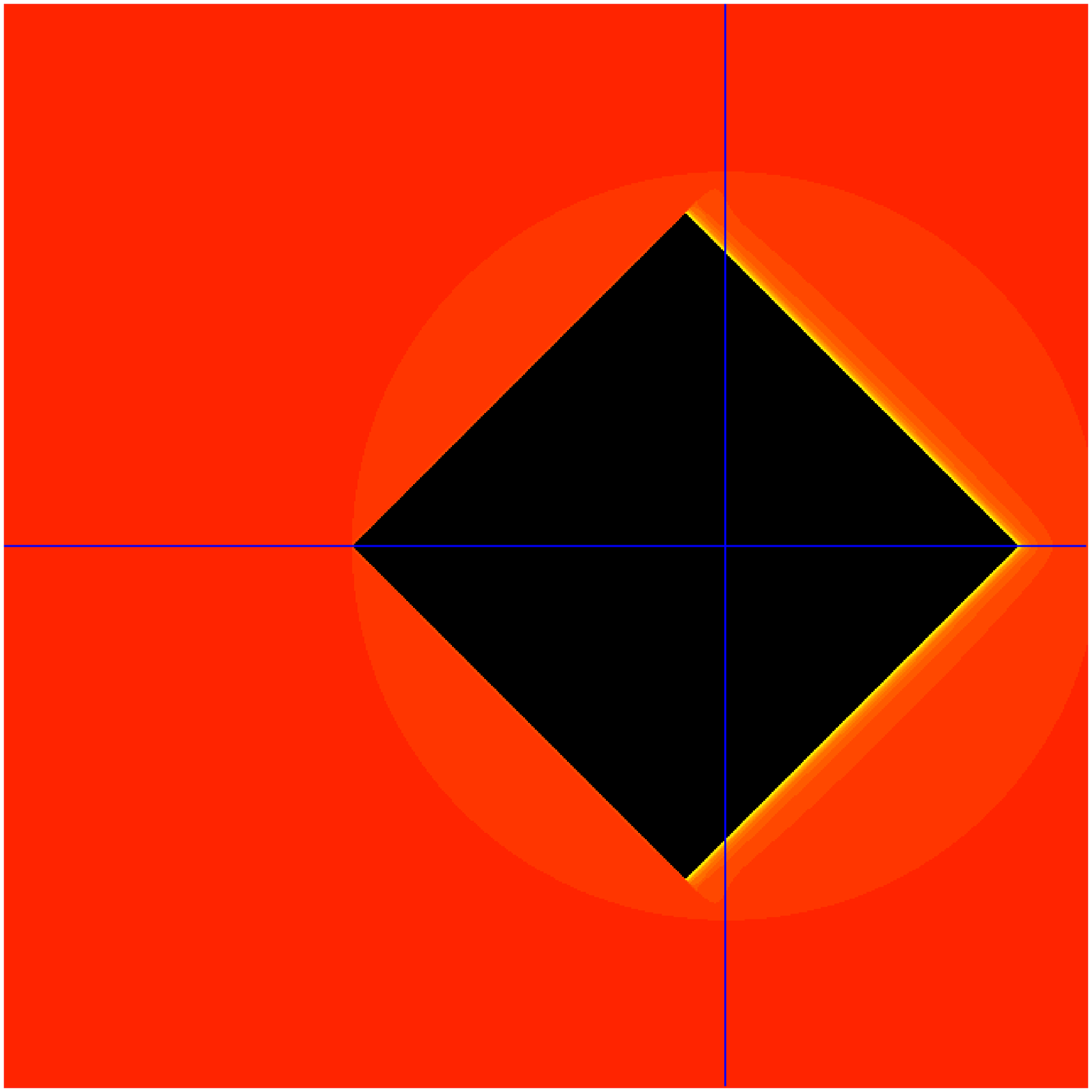}}
\subfigure[$\mHyb^{30}$, $\lc -2, 1 \rc \times \lc \frac{-3}{2}, \frac{3}{2} \rc$]{
	\includegraphics[scale = 0.15]{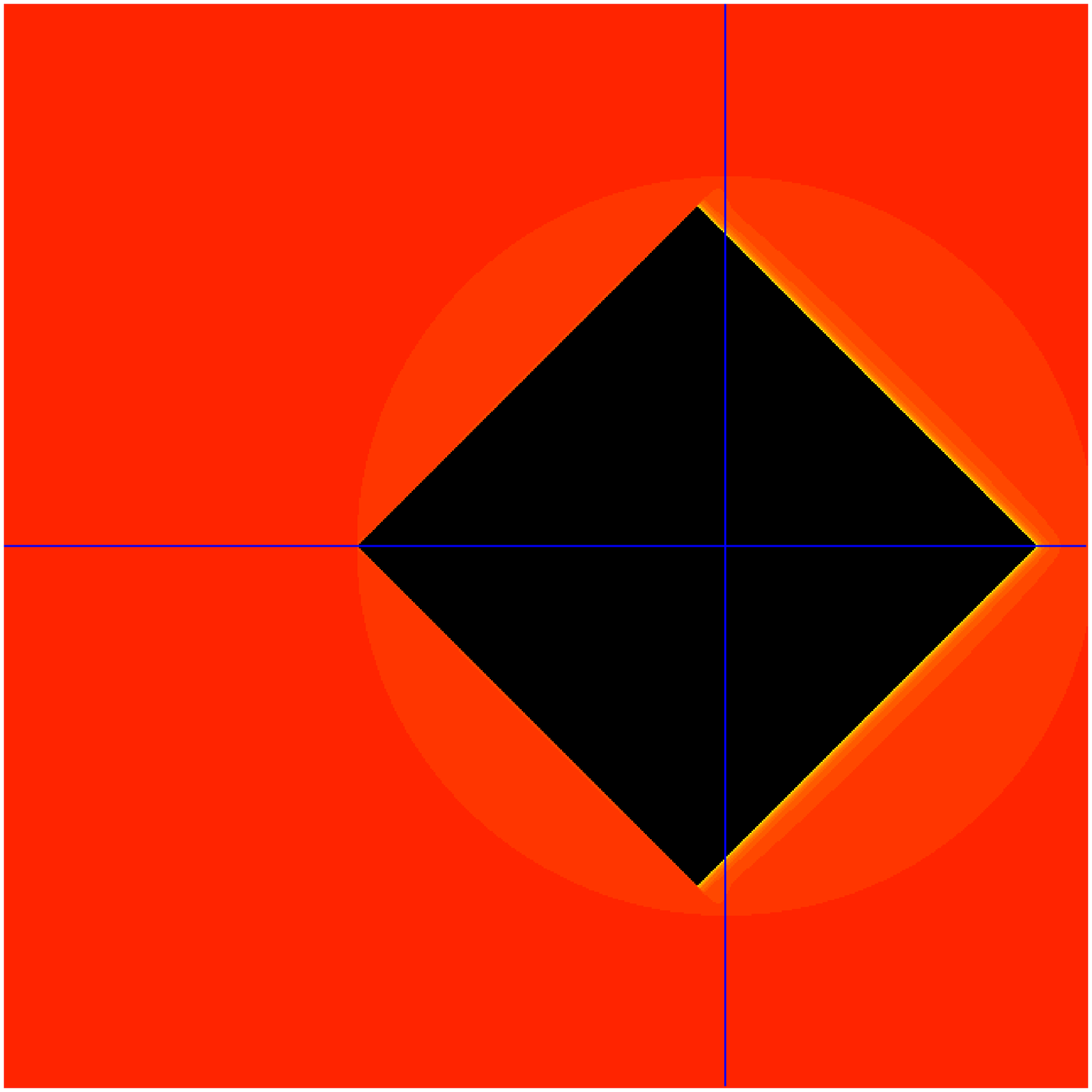}}
\caption{Hyperbrots for several even integers. The black color represents the Hyperbrots.}\label{figureHyperbrots}
\end{figure}

Let $\mathcal{S}(\mR^2)$ be the collection of non-empty compact subsets of $\mR^2$ and define the distance between $A$ and $B$ in $\mathcal{S}(\mR^2)$ as
	\begin{align*}
	d(A,B)&:= \max_{x \in A} \oa d(x,B) \fa \\
	&= \max_{x \in A} \oa \min_{y \in B} \oa \Vert x - y \Vert \fa \fa
	\end{align*}
where $\Vert \cdot \Vert$ is the Euclidean distance on $\mR^2$. Now, let $(\mathcal{S}(\mR^2), h)$ be the so-called \textbf{Fractals metric space} on $\mR^2$ where $\mathcal{S}(\mR^2)$ is the collection of all non-empty compact subsets of $\mR^2$ and $h: \, \mathcal{S}(\mR^2) \times \mathcal{S}(\mR^2) \ra [0,+\infty )$ is the Hausdorff distance on the collection $\mathcal{S}(\mR^2)$ defined as $h(A,B) := \max \oa d(A,B) , d(B,A) \fa$. With these definitions from \cite{Barnsley}, we have the following theorem.

	\begin{theorem}\label{HausdorffH2D}
	 Let $\mHyb := \oa x+y\bj\in\mD \, : \, |x| + |y| \leq 1\fa.$ Then
		\begin{align*}
		\lim_{n \ra \infty} h(\mHyb , \mHyb^{2n}) = 0 \text{.}
		\end{align*}
	\end{theorem}
	
	\begin{proof}
	Let $p \geq 2$ be an even integer. We start by computing $d(\mHyb , \mHyb^p)$. Let $c:=a+b\bj \in \mHyb$. We now have to find $d(c, \mHyb^p)$. If $c \in \mHyb^p$, then we have $d(c, \mHyb^p) = d(c,c) = 0$. If $c \not \in \mHyb^p$, then it is easy to see that the minimum occurs on the boundary of $\mHyb^p$. It is also obvious that the maximum of $d(c, \mHyb^p)$ is attained on the boundary of $\mHyb$. Since the sides of $\mHyb$ and $\mHyb^p$ are parallel, and $t_p+\frac{l_2}{2}=\frac{p-1}{p^{p/(p-1)}}<1$, we have that $d(c, \mHyb^p)$ attains its maximum at the right corner of $\mHyb$ and $\mHyb^p$, that is
		\begin{align*}
		d(\mHyb, \mHyb^p) = |(t_p+\frac{l_p}{2})-1|. 
		\end{align*}
	Similarly, we compute
		\begin{align*}
		d(\mHyb^p , \mHyb) = |(t_p-\frac{l_p}{2})+1|. 
		\end{align*}
	Then,
		\begin{align*}
		h(\mHyb , \mHyb^p) = \max \oa |(\frac{l_p}{2}+t_p)-1|, |(t_p-\frac{l_p}{2})+1|\fa
		\end{align*}
	for any even integer $p \geq 2$.	A simple computation shows that
		\begin{align*}
		\lim_{p \ra \infty} t_p = 0 \text{, } \quad \quad \lim_{p \ra \infty} 1-\frac{l_p}{2} = 0.		
		\end{align*}
	Thus, 
		\begin{align*}
		\lim_{n \ra \infty} h(\mHyb , \mHyb^{2n}) = 0 \text{.} 
		\end{align*}
	This completes the proof. \hfill $\qed$
	\end{proof}
	
	\subsection{Characterization of the generalized Perplexbrot}
	In this subsection, we generalize Hyperbrots in three dimensions. 
	
	We begin by recalling the definition of the Multibrot sets for tricomplex numbers (see \cite{RochonParise} and \cite{PariseRochon2}).
	\begin{definition}\label{d3.1}
	Let $Q_{p,c}(\eta )=\eta^p+c$ where $\eta , c \in \mM (3)$ and $p\geq 2$ an integer. The tricomplex Multibrot set of order $p$ is defined as the set
		\begin{equation}
		\mMan_3^p:=\oa c \in \mM (3) \, : \, \Qpit \text{ is bounded } \fa \text{.}
		\end{equation}
	\end{definition}
Moreover, a \textbf{principal 3D slice} of a $\mMan_3^p$ set is defined as follows
	\begin{equation}\label{eq5.2.1}
	 \mTet^p(\bim , \bk , \bil ) :=\oa c \in \mT (\bim , \bk , \bil ) \, : \, \oa Q_{p,c}^m(0) \fa_{m=1}^{\infty} \text{ is bounded } \fa
\text{.}
	\end{equation}
Let us adopt the same notation as in \cite{GarantRochon}, \cite{RochonParise} and \cite{PariseRochon2} for the generalized Hyperbrot, called the \textit{Perplexbrot},
	\begin{align}
	\mathcal{P}^p:=\mathcal{T}^p(1,\bjp,\bjd)
=\oa c=c_0+c_5\bjp+c_6\bjd \, : \, c_i \in \mathbb{R} \text{ and } \right. \notag\\
\left. \oa Q_{p,c}^m(0)\fa_{m=1}^{\infty} \text{ is bounded} \fa \text{.}\label{eq3.3.1}
	\end{align}
In \cite{PariseRochon2}, it is proved that $\cP^p$ is a regular octahedron for all odd integers $p > 2$. We want to extend the result to all even integers $p \geq 2$.

First, we need this following lemma.

	\begin{lemma}\label{lem3.3.1}
	We have the following characterization of the generalized Perplexbrot
		\begin{equation*}
		\mathcal{P}^p=\bigcup_{y\in \left[-\frac{l_p}{2},\frac{l_p}{2}\right]} \oa \left[ (\mathcal{H}^p-y\bjp)\cap (\mathcal{H}^p+y\bjp)\right] + y\bjd\fa
		\end{equation*}
	where $\mathcal{H}^p$ is the Hyperbrot for an even integer $p \geq 2$.
	\end{lemma}

	\begin{proof}
	The proof is similar to the one in \cite{PariseRochon2}. We just replace $m_p$ by $l_p$ in the proof. \hfill $\qed$
	\end{proof}

As a consequence of Lemma \ref{lem3.3.1} and Theorem \ref{t3.2}, we have the following corollary illustrated by Figure \ref{figureHyperbrots2}:

	\begin{corollary}\label{theo3.2.5}
	$\mathcal{P}^p$ is a regular octahedron of edges $\frac{\sqrt{2}}{2}l_p$ where $p\geq2$ is an even integer. Moreover, the generalized Perplexbrot can be rewritten as the set
		\begin{equation}
		\mathcal{P}^p = \oa x + y\bjp + z \bjd \, : \, (x,y,z) \in \mR^3 \, \text{ and } \, |x - t_p| + |y| + |z| \leq \frac{l_p}{2} \fa.
		\end{equation}
	\end{corollary}

Figure \ref{figureHyperbrots2} shows some Perplexbrots in 3-dimensional space. We remark that they are centered at the point $(t_p, 0, 0)$.
\begin{figure}
\centering
\subfigure[$\cP^2$, $\lc -2,2\rc^3$]{
	\includegraphics[scale = 0.25]{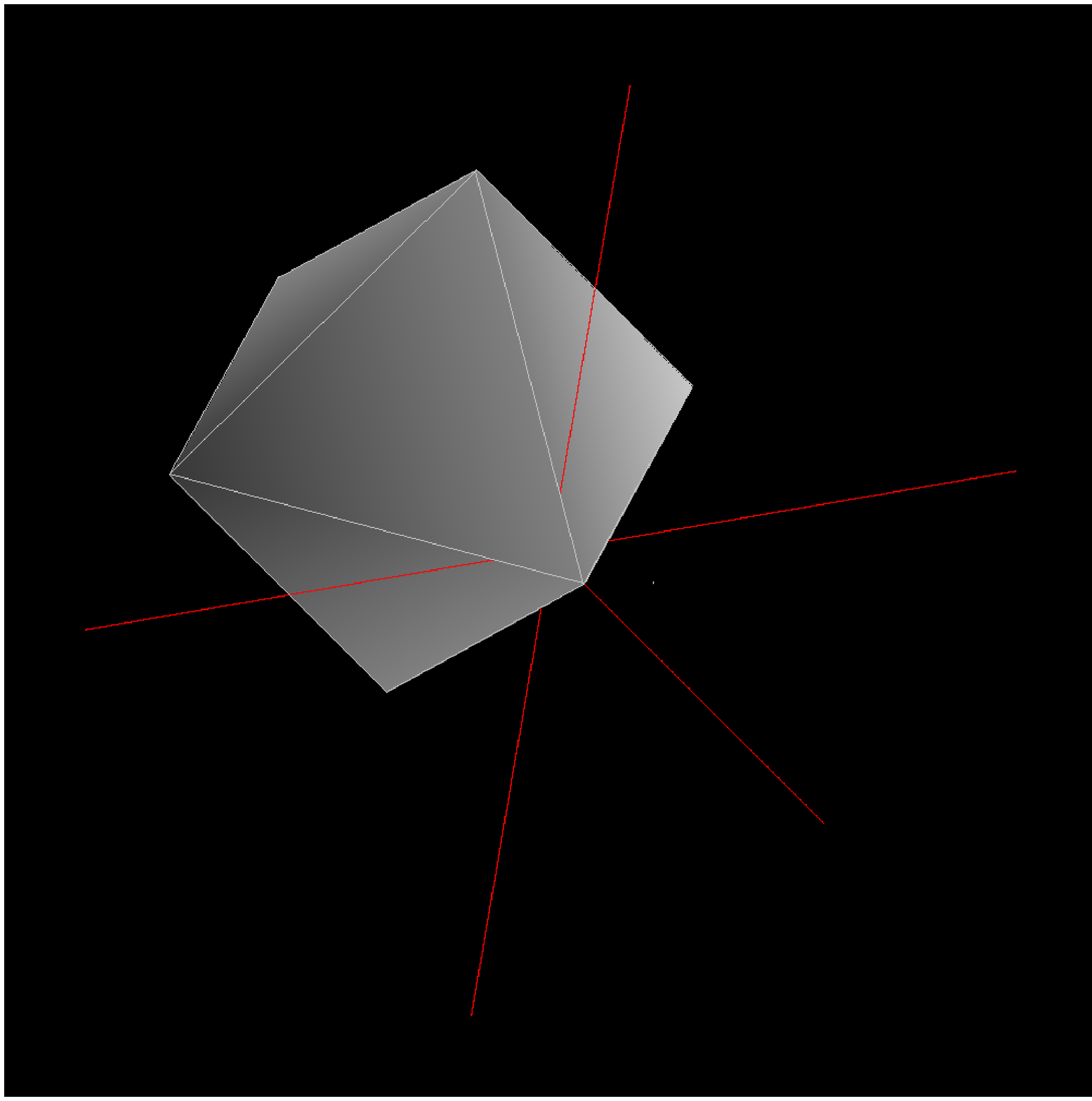}}
\subfigure[$\cP^{12}$, $\lc -2,2\rc^3$]{
	\includegraphics[scale = 0.25]{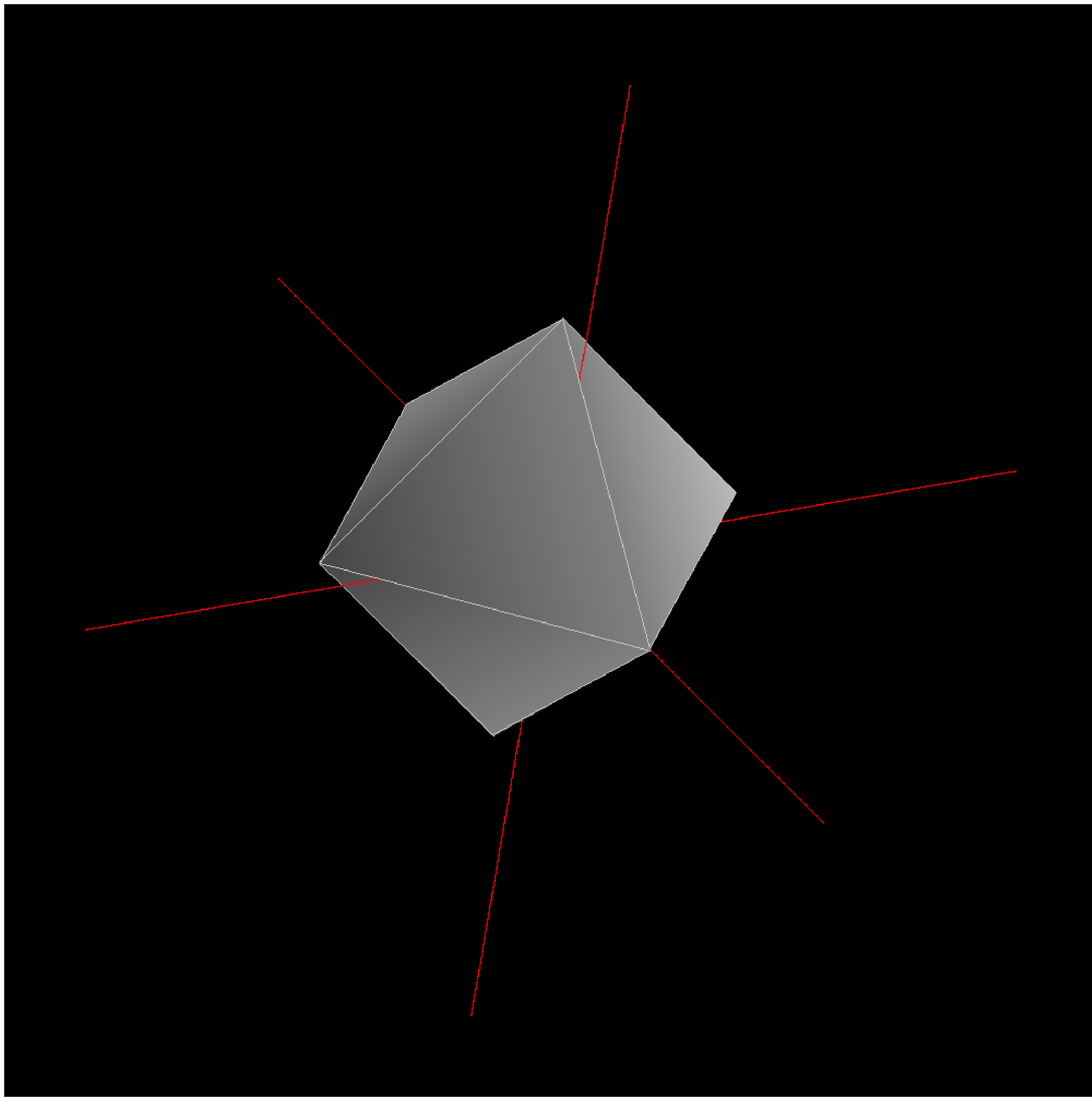}}
\subfigure[$\cP^{20}$, $\lc -2,2\rc^3$]{
	\includegraphics[scale = 0.25]{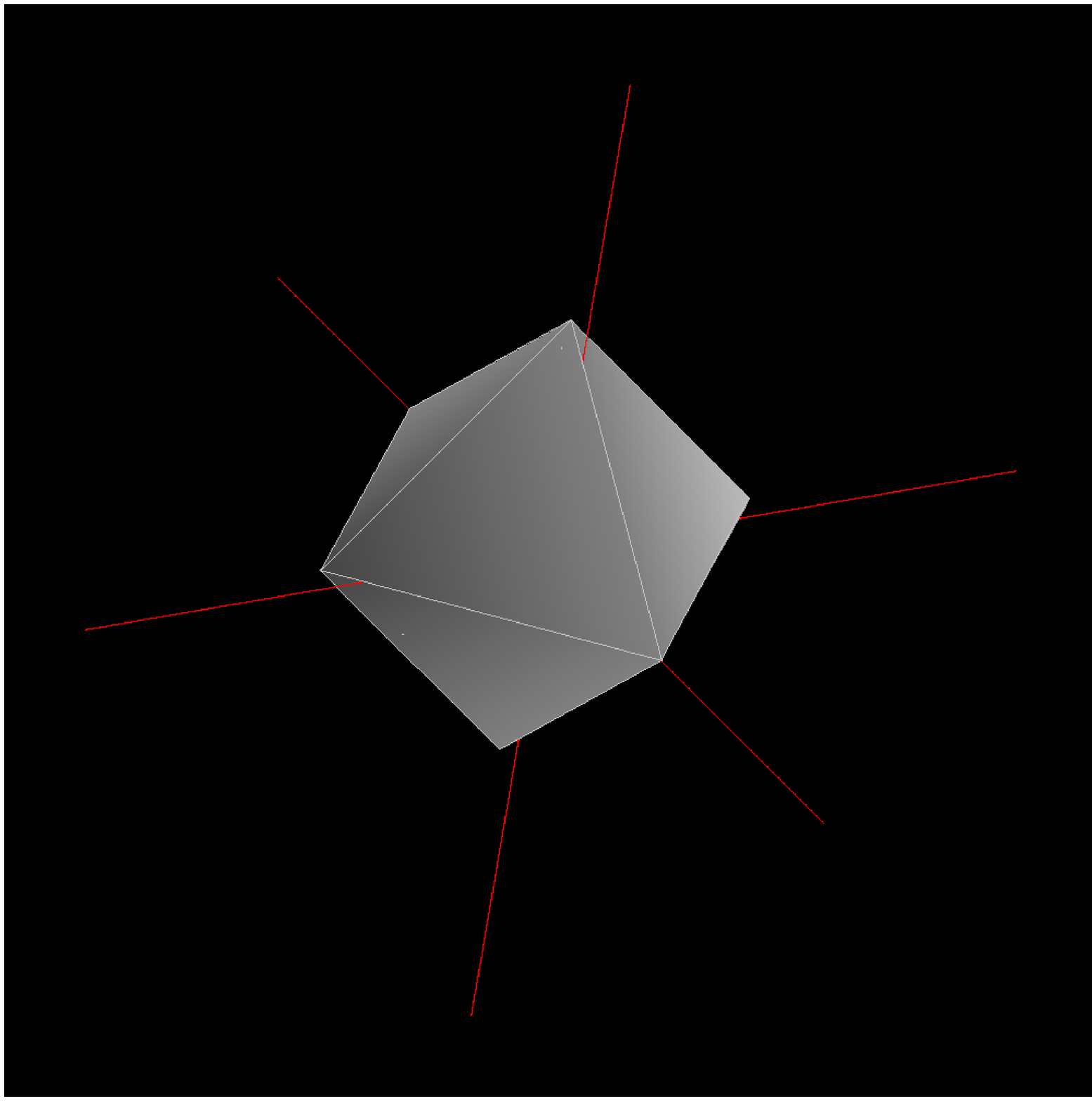}}
\subfigure[$\cP^{30}$, $\lc -2,2\rc^3$]{
	\includegraphics[scale = 0.25]{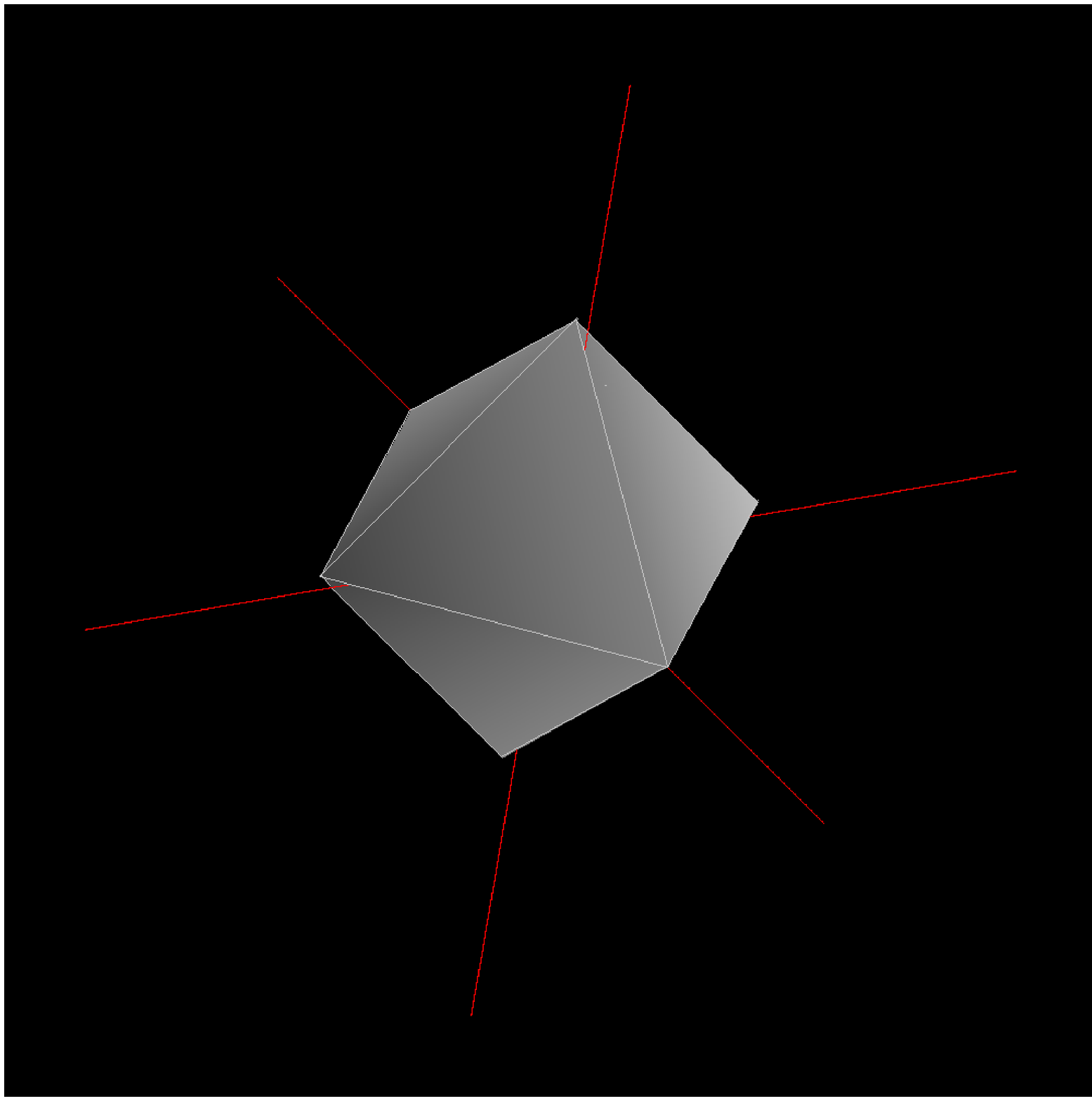}}	
\caption{Perplexbrots for several even integers}\label{figureHyperbrots2}
\end{figure}

The next result is similar to Theorem \ref{HausdorffH2D}, that is if $ p \ra \infty$, then the sequence of generalized Perplexbrots converges to a non-empty compact subset of $\mR^3$.
	\begin{theorem}
    Let $\cP := \oa x + y \bjp + z\bjd : (x,y,z)\in \mR^3 \text{ and }  |x| + |y| + |z| \leq 1 \fa\text{.}$
	Then 
		\begin{align*}
		\lim_{n \ra \infty} h(\cP , \cP^{2n}) = 0 \text{.}
		\end{align*}
	\end{theorem}
	
	\begin{proof}
	The proof is similar to the two-dimensional case. \hfill $\qed$
	\end{proof}
	
\section*{Conclusion}
In this article, we treated Multibrot sets for polynomial of even degrees. The characterization that we obtained for the intersection of the real line with a Multibrot set implies that the Hyperbrots and the Perpexbrots for polynomial of even degrees are squares and regular octahedrons respectively.

This work concludes a sequence of previous works on the same topic. If we join all the results of this work with results of \cite{RochonParise} and \cite{PariseRochon2}, we have finally proved the following theorems.
\begin{theorem}
Let $\mManp$ be the Mandelbrot set for the polynomial $Q_{p,c}(z)=z^p+c$ where $z,c \in \mC$ and $p\geq 2$ an integer. Then, we have two cases for the intersection $\mManp \cap \mR$:
	\begin{itemize}
	\item[i.] If $p$ is even, then $\mManp \cap \mR =\left[-2^{\frac{1}{p-1}},(p-1)p^{\frac{-p}{p-1}} \right]$;
	\item[ii.] If $p$ is odd, then $\mManp \cap \mR=\left[ -(p-1)p^{\frac{-p}{p-1}} , \, (p-1)p^{\frac{-p}{p-1}} \right]$.
	\end{itemize}
\end{theorem}
	
\begin{theorem}
The Hyperbrot $\mHyb^p$ is a square for any integer $p \geq 2$ and the sequence $\oa \mHyb^p \fa_{p=2}^{\infty}$ converges to a square with unit half-diagonal. Moreover, the Perplexbrot is a regular octahedron for any integer $p \geq 2$ and the sequence $ \oa \cP^{p} \fa_{p=2}^{\infty}$ converges to a regular octahedron with unit half-diagonal.
\end{theorem}

\section*{Acknowledgement}
DR is grateful to the Natural Sciences and Engineering Research Council of Canada (NSERC) for financial support. TR thanks NSERC and the Canada Research Chairs Program for financial support. POP would also like to thank the NSERC for the award 
of a graduate research grant.


\end{document}